\documentclass[12pt]{amsart}
\usepackage{amssymb}

\newtheorem{thm}     {Theorem}
\newtheorem{prop}    [thm]{Proposition}

\newtheorem{lemma}   [thm]{Lemma}

\newcommand{\Hi}{{\mathbb H}}
\newcommand{\B}{\mathbb B}
\newcommand{\C}{\mathbb C}
\newcommand{\R}{\mathbb R}
\newcommand{\T}{\mathbb T}
\newcommand{\Z}{\mathbb Z}

\def\Re{{\rm Re\,}}
\def\Im{{\rm Im\,}}

\def\<{\langle}
\def\>{\rangle}
\def\({\left(}
\def\){\right)}

\def\Aut{{\rm Aut}}

\def\alert{}

\begin{document}


\title[Non-linear Lie groups that can be realized as automorphism groups of bounded domains]
{Non-linear Lie groups that can be realized \\
as automorphism groups of bounded domains}

\author{George Shabat}
\address{Russian State University for the Humanities,
Moscow, GSP-3, 125267, Russia}
\email{george.shabat@gmail.com}

\author{Alexander Tumanov}
\address{Department of Mathematics, University of Illinois,
1409 West Green St., Urbana, IL 61801, U.S.A.}
\email{tumanov@illinois.edu}

\maketitle

\begin{abstract}
We consider a problem whether a given Lie group can be realized as the group of all biholomorphic automorphisms of a bounded domain in $\C^n$.
In an earlier paper of 1990, we proved the result for connected \emph{linear} Lie groups.
In this paper we give examples of non-linear groups for which the result still holds.

\end{abstract}

\section{Introduction}

Let $D\subset\C^n$ be a bounded domain. H. Cartan \cite{Cartan} proved that the group $\Aut(D)$ of all biholomorphic automorphisms of $D$ is a (real finite dimensional) Lie group.
Is the converse true? In other words, which Lie groups can be realized as $\Aut(D)$ for a bounded domain $D\subset\C^n$?

Bedford and Dadok \cite{Bedford} and Saerens and Zame \cite{Saerens} proved that every \emph{compact} Lie group can be realized as $\Aut(D)$ for a bounded strongly pseudoconvex domain $D\subset\C^n$.
On the other hand,
Wong \cite{Wong} and Rosay \cite{Rosay} proved that if $D\subset\C^n$ is bounded, strongly pseudoconvex, and $\Aut(D)$ is \emph{not compact}, then $D$ is biholomorphically equivalent to the unit ball $\B^n\subset\C^n$.
Therefore, if the group is not compact, we cannot expect to realize it as $\Aut(D)$ for a bounded strongly pseudoconvex domain $D\subset\C^n$.

A Lie group is called \emph{linear} if it is isomorphic to a subgroup of a general linear group $GL(n,\R)$ of all real
nonsingular $n\times n$ matrices.

We call a domain $D\subset\C^n$ a domain of \emph{bounded type} if $D$ is biholomorphically equivalent to a bounded domain.

In an earlier paper \cite{Shabat}, we proved that every (possibly non-compact) connected \alert{linear} Lie group can be realized as $\Aut(D)$, where $D\subset\C^n$ is a strongly pseudoconvex domain of bounded type.
Winkelmann \cite{Winkelmann} and Kan \cite{Kan} proved that every connected (possibly non-linear) Lie group can be realized as $\Aut(D)$, where $D$ is a \emph{complete hyperbolic Stein manifold}.

The question whether $D$ can be chosen a \emph{bounded domain} in $\C^n$ has remained open so far.

Recall $SL(n,\R)$ denotes the group of all real $n\times n$ matrices with determinant 1. We consider connected Lie groups locally isomorphic to $SL(2,\R)$.
Among these groups only $SL(2,\R)$ itself and $PSL(2,\R):=SL(2,\R)/\{\pm I\}$ are linear (see \cite{Onishchik}).
The rest are typical examples of \emph{non-linear} Lie groups. In particular, $\widetilde{SL}(2,\R)$, the universal cover of $SL(2,\R)$ is non-linear.
Our main result is the following.

\begin{thm}\label{Main}
Let $G$ be a connected Lie group locally isomorphic to $SL(2,\R)$. Then there exists a strongly pseudoconvex domain $D$ of bounded type in $\C^4$ such that $\Aut(D)$ is isomorphic to $G$.
\end{thm}

In the end of the paper, we give another example of a non-linear group for which a similar result holds.

\section{General results}
Recall that a \emph{group action} $G:X$ of a group $G$ on a set $X$ is a mapping $G\times X \to X$, which we denote as $(g,x)\mapsto gx$, such that
$e x=x$ and $g_1 (g_2 x)=(g_1g_2) x$.
Here $e\in G$ is the identity.

A group action $G:X$ is \emph{free} (or with no fixed points) if for every $x\in X$, the map $G\to X$, $g\mapsto g x$ is injective.

A group action $G:X$ is \emph{proper} if the mapping $G\times X \to X\times X$, $(g,x)\mapsto (g x,x)$, is proper.
Here $G$ is a topological group, $X$ is a topological space, and the action $G\times X \to X$ is continuous.

A group action $G:X$ is \emph{holomorphic} if for every $g\in G$, the map $x\mapsto gx$ is holomorphic. Here $X$ is a complex manifold.

\begin{prop}\cite{Bedford, Saerens, Shabat, Winkelmann}
\label{Common}
Let $G:\Omega$ be a holomorphic group action of a connected Lie group $G$ on a domain $\Omega\subset\C^n$. Suppose the action is proper, free, and the orbits are totally real. Then a generic smooth small tubular $G$-invariant neighborhood $D$ of each orbit is strongly pseudoconvex, and $\Aut(D)$ is isomorphic to $G$.
\end{prop}

The proof consists of two steps. In the first step, one proves that every $f\in \Aut(D)$ extends smoothly to the most of the boundary $bD$. In the second step, using local invariants of CR structure of $bD$ \cite{Chern-Moser}, by small perturbations, one can rule out automorphisms other than the ones induced by the action of $G$.

If $G$ is compact, then $D$ is a bounded strongly pseudoconvex domain, and the smooth extension follows by Fefferman's theorem.
In the case that $G$ is not compact, our short paper \cite{Shabat} did not include full details of the first step. The proof can be found in \cite{Winkelmann}.

Let $G$ be a Lie group that can be realized using Proposition \ref{Common}. We describe a situation in which covering groups also can be realized.

\begin{prop}\label{Cover}
Let $G$ be a connected Lie group with $\pi_1(G)=\Z$.
Let $G:\Omega\subset\C^n$ be a holomorphic free proper action with totally real orbits in a domain $\Omega$ of bounded type.
Let $\phi:\Omega\to\C^*=\C\setminus\{0\}$ be a holomorphic function such that $|\phi|>\epsilon$, here $\epsilon>0$ is constant.
Let $M$ be an orbit. Suppose $\phi$ induces an isomorphism $\phi_*:\pi_1(M)\to \Z=\pi_1(\C^*)$.
Let $G_k$ be a $k$-sheeted covering group for $G$, $1<k\le\infty$. Then there is a strongly pseudoconvex domain $D\subset\C^{n+1}$ of bounded type with $\Aut(D)=G_k$.
\end{prop}

\begin{proof}
Consider the case $k=\infty$. Then $G_\infty=\tilde G$ is the universal cover of $G$.
Shrinking $\Omega$ if necessary, we assume
$\Omega$ is a small $G$-invariant neighborhood of $M$. Then the universal cover $\tilde\Omega$ is the graph of $\log\phi$ over $\Omega$.
$$
\tilde{\Omega}=\{(z,\log\phi(z)): z\in\Omega\}\subset\C^{n+1},
$$
here all values of the logarithm are used.
The action $G:\Omega$ lifts to an action
$\tilde G:\tilde \Omega$.
Indeed, let $\tilde g\in \tilde G$ be represented by a curve
$\tilde g:[0,1]\to G$ with
$\tilde g(0)=e$, $\tilde g(1)=g$. Then we define
\begin{equation*}
\tilde g(z,\log\phi(z))=(g z, \log\phi(\tilde g z)),\quad
\tilde g\in \tilde G,\quad
z\in\Omega,
\end{equation*}
here $\log\phi(\tilde g z)=\gamma(1)$ for a continuous curve
$\gamma(t)=\log\phi(\tilde g(t) z)$ with value
$\gamma(0)=\log\phi(z)$ used in the left hand side.
With some abuse of notation, we define
$$
\tilde{\tilde\Omega}
=\{(z,w): z\in\Omega, |w-\log\phi(z)|<1 \}\subset\C^{n+1},
$$
that is, for each $z\in\Omega$, the set
$\{w\in\C: (z,w)\in\tilde{\tilde\Omega} \}$
is the union of all unit discs with centers at all values
of $\log\phi(z)$.
The action $\tilde G:\tilde \Omega$ extends to
$\tilde G:\tilde{\tilde\Omega}$ as follows.
$$
\tilde g(z,w)=(g z,w-\log\phi(z)+\log\phi(\tilde g z)), \quad
\tilde g\in \tilde G, \quad
(z,w)\in \tilde{\tilde\Omega}.
$$
Here $\log\phi(z)$ is the value satisfying $|w-\log\phi(z)|<1$,
and $\log\phi(\tilde g z)$ is the same as above.

Since $|\phi|>\epsilon$,
we have $\Re(\log\phi)>\log\epsilon$.
Then $\tilde{\tilde\Omega}$ is a domain of bounded type.

The action $\tilde G:\tilde{\tilde\Omega}$ is free, proper, and the orbits are totally real. Hence the conclusion follows by Proposition \ref{Common}.

For the group $G_k$ with $k<\infty$, the proof goes along the same lines with $\phi^{1/k}$ in place of $\log\phi$. We leave the details to the reader.
\end{proof}

\section{Proof of main result}
We apply the results of the previous section to groups locally isomorphic to $SL(2,\R)$.

Let $G=PSL(2,\R):=SL(2,\R)/\{\pm I\}$.

Let $\Hi=\{z\in\C: \Im z>0\}$ be the upper half-plane.

Then $G:\Hi$ by fractional-linear transformations as follows.

\begin{equation*}
g z=\frac{az+b}{cz+d},\quad
g=\pm\begin{pmatrix}
         a & b \\
         c & d
       \end{pmatrix}\in G,\quad
z\in\Hi.
\end{equation*}

Define $G:\Hi^3\subset\C^3$,
\begin{equation*}
g (z_1,z_2,z_3)=(g z_1,g z_2,g z_3).
\end{equation*}

On the subset of all triples with distinct components, this action is free, proper, and the orbits are totally real.
We now look for a function $\phi$ for the action $G:\Hi^3$ to apply Proposition \ref{Cover}.

The group $G$ also acts on the complexification $G^c=PSL(2,\C)$, which is not a domain in $\C^n$.
Fix $\zeta=(\zeta_1,\zeta_2,\zeta_3)\in\Hi^3$ with distinct components. Define a map
\begin{align*}
& \Phi:G^c\to \C^3, \\
& \Phi:G^c\ni h=\pm\begin{pmatrix}
         a & b \\
         c & d
       \end{pmatrix}\mapsto
h\zeta=(h\zeta_1,h\zeta_2,h\zeta_3)\in\C^3.
\end{align*}
The map $\Phi$ is holomorphic, injective, and commutes with the actions $G:\C^3$ and $G:G^c$.
The map $\Phi$ reduces the construction of $\phi$ to
$G^c=PSL(2,\C)$.

We consider $G\subset G^c$ as the orbit of the identity matrix $I$.
We need a holomorphic function $\phi:G^c\to\C^*$
such that $\phi_*:\pi_1(G)\to\pi_1(\C^*)=\Z$ is an isomorphism.

For
$g=\begin{pmatrix}
         a & b \\
         c & d
\end{pmatrix}\in SL(2,\C)$, we preliminary define $\phi(g)=a+ic$.
Then $\phi: SO(2,\R)\to \T$ is an isomorphism, here
$SO(2,\R)\subset SL(2,\R)$ is the group of all real orthogonal matrices with
determinant 1, $\T\subset\C$ is the unit circle.
Hence
$\phi_*:\pi_1(SL(2,\R))\to\Z$ is an isomorphism.

For $g=\pm\begin{pmatrix}
         a & b \\
         c & d
\end{pmatrix}\in G^c=PSL(2,\C)$, we can define $\phi(g)=(a+ic)^2$.
Then
$\phi_*:\pi_1(G)\to\Z$ is again an isomorphism.
However, one can see that this function
$\phi$ has zeros in any $G$-invariant neighborhood of $I$.

Finally, for $g=\pm\begin{pmatrix}
         a & b \\
         c & d
\end{pmatrix}\in G^c$,
we define
$$
\phi(g)=\frac{1}{4} ((a+d)+i(c-b))^2.
$$
This function $\phi$ coincides with the previous version
on the orthogonal group. Then
$\phi_*:\pi_1(G)\to\Z$ is again an isomorphism.
This function $\phi$ is bounded away from 0 on a $G$-invariant neighborhood of $I$ according to the following lemma.

\begin{lemma}
There exist $\epsilon>0$ and $\delta>0$ such that for every $g\in G$ and $h\in G^c$, $|h-I|<\delta$ implies $|\phi(gh)|>\epsilon$.
\end{lemma}

\begin{proof}
It suffices to prove the lemma for $G=SL(2,\R)$ instead of $PSL(2,\R)$ and
$\psi(g)=(a+d)+i(c-b)$ instead of $\phi$.

We claim that for some small $\epsilon>0$, if
$g\in SL(2,\C)$ and
$|\psi(g)|\le\epsilon$, then $|\Re g|\le 2|\Im g|$. Here $\Re$ and $\Im$ are applied to each entry of $g$, and
$|g|^2=|a|^2+|b|^2+|c|^2+|d|^2$ is the Euclidean norm.
We have
$$
(a+d)=i(b-c)+\psi, \quad
|\psi|\le \epsilon.
$$
By squaring both parts and using $\det g=1$, we obtain
\begin{align*}
& a^2+b^2+c^2+d^2+2=2i(b-c)\psi+\psi^2, \\
& \Re(a^2+b^2+c^2+d^2)+2\le 2\epsilon |b-c|+\epsilon^2, \\
& |\Re g|^2-|\Im g|^2+2\le
4\epsilon(|\Re g| + |\Im g|)+\epsilon^2.
\end{align*}
By applying the inequality
$4u\epsilon\le \frac{3}{5}u^2+\frac{20}{3}\epsilon^2$, we obtain
$$
|\Re g|^2\le 4|\Im g|^2+5\left(\frac{43}{6}\epsilon^2-1\right).
$$
Hence the claim holds, say for $\epsilon=1/3$.

We now prove that if for some small constant $\delta>0$,
$g\in G$, $h\in G^c$, $|h-I|<\delta$, then $|\psi(gh)|>\epsilon$. Suppose otherwise
$|\psi(gh)|\le\epsilon$. Then by the above claim,
$|\Re (gh)|\le 2|\Im (gh)|$.

Since $\Im g=0$, we have
$|\Im(gh)|=|\Im(g(h-I))|\le |g|\delta$.
We also have
$|\Re(gh)|=|g+\Re(g(h-I))|\ge |g|-|\Re(g(h-I))|\ge |g|(1-\delta)$.

Combining the above inequalities, we have
$|g|(1-\delta)\le 2|g|\delta$. Since $g\ne 0$, we get
$1-\delta\le 2\delta$, $\delta\ge 1/3$.
Hence, for  $\delta< 1/3$, we obtain the desired conclusion.
\end{proof}

This lemma concludes the proof of the main result.

\section{Another example}

We give another example of a non-linear Lie group $G$ that can be realized as $\Aut(D)$ for a bounded domain $D$.
Following \cite{Onishchik},
we introduce $G$ as a quotient of the Heisenberg group $\tilde G$ as follows.
$$
\tilde G=\left\{
g= \begin{pmatrix}
1 & a & c \\
0 & 1 & b \\
0 & 0 & 1
       \end{pmatrix} : a,b,c\in\R \right\}, \quad
H=\{g\in\tilde G: a=b=0, c\in\Z \}, \quad
G=\tilde G/H.
$$
The group $G$ is non-linear (see \cite{Onishchik}).
We describe it directly as
\begin{align*}
& G=\R\times\R\times\T, \\
& (a,b,c)(x,y,z)=(a+x, b+y, cz e^{iay}).
\end{align*}
Here $\T$ denotes the unit circle in $\C$.
The group $G$ has the obvious complexification
$$
G^c=\C\times\C\times\C^*, \quad
\C^*=\C\setminus\{0\}.
$$
The group $G$ acts on $G^c$ by left translations.
We claim that there is a $G$-invariant domain $\Omega$ of bounded type, and by Proposition \ref{Common} there is a domain $D\subset\C^3$ such that
$\Aut(D)$ is isomorphic to $G$.

Indeed, let $\Omega=GU$,
$$
U=\{(x,y,z)\in G^c: |x|<1, |y|<1, |z|<2\}.
$$
We show that $\Omega$ is of bounded type.
Let $(u,v,w)\in \Omega$. Then
$$
u=a+x, \quad
v=b+y, \quad
w=cze^{iay}, \quad
(a,b,c)\in G, \quad
(x,y,z)\in U.
$$
Then
$$
|\Im u|<1,\quad
|\Im v|<1,\quad
|w|=|z|e^{\Re(iay)}<2e^{|a|}.
$$
Since $|\Im u|<1$, we have $|a|\le |u|+1$, hence there exists a constant $C>0$ such that
$|w|<C|e^{u^2}|$.

Put $w'=w+2Ce^{u^2}$.
By increasing $C$ if necessary, we have
$$
|w'|\ge 2C|e^{u^2}|-|w|>C|e^{u^2}|>1.
$$

Then $(u,v,w)\mapsto (u,v,w')$ biholomorphically maps $\Omega$ to a domain of bounded type, as desired.

\end{document}